\documentclass[12pt]{amsart}

\usepackage{hyperref}
\hypersetup{nesting=true,debug=true,naturalnames=true}
\usepackage{graphicx,amssymb,upref,units}

\usepackage{tikz}

\let\<\langle
\let\>\rangle

\let\uml\"

\usepackage{amsmath,amssymb,fullpage,xy,mathrsfs,amsthm,xcolor,amsxtra,graphicx,verbatim}
\usepackage{hyperref}
\hypersetup{colorlinks=true,urlcolor=magenta,citecolor=magenta,linkcolor=blue}
\usepackage{enumerate}
\usepackage{colonequals}
\usepackage{multirow}

\usepackage[mathcal]{euscript}
\newcommand{\Gal}{\operatorname{Gal}}

\newcommand{\p}{\mathfrak{p}}
\newcommand{\Fp}{\mathbf{F}_{\mathfrak{p}}}
\newcommand{\OK}{\mathcal{O}_{K}}

\newcommand{\Fl}{\mathbf{F}_{\ell}}

\newcommand{\Aut}{\operatorname{Aut}}

\newcommand{\im}{\operatorname{im}}

\newcommand{\Z}{\mathbf{Z}}
\newcommand{\F}{\mathbf{F}}

\newcommand{\Q}{\mathbf{Q}}

\newcommand{\GSp}{\operatorname{GSp}}
\newcommand{\PSL}{\operatorname{PSL}}

\newcommand{\PGL}{\operatorname{PGL}}
\newcommand{\PSU}{\operatorname{PSU}}

\newcommand{\Sp}{\operatorname{Sp}}
\newcommand{\GL}{\operatorname{GL}}
\newcommand{\SL}{\operatorname{SL}}

\newcommand{\AGL}{\operatorname{AGL}}

\newcommand{\Jac}{\operatorname{Jac}}
\newcommand{\C}{\mathbf{C}}

\newcommand{\St}{\mathsf{St}}

\newcommand{\disc}{\operatorname{disc}}

\numberwithin{equation}{subsection}

\theoremstyle{plain}
\newtheorem{thm}[equation]{Theorem}
\newtheorem{conjecture}[equation]{Conjecture}

\newtheorem*{qu}{Question}

\theoremstyle{remark}
\newtheorem{rmk}[equation]{Remark}

\newtheorem{exm}[equation]{Example}

 \input xy
 \xyoption{all}

\begin{document}

\title{Realizations of Unisingular Representations by Hyperelliptic Jacobians}

\author{John Cullinan}
\address{Department of Mathematics, Bard College, Annandale-On-Hudson, NY 12504, USA}
\email{cullinan@bard.edu}
\urladdr{\url{http://faculty.bard.edu/cullinan/}}

\begin{abstract}
A representation of a finite group $G$ on a finite dimensional vector space $V$ is called \textbf{unisingular} if every $g\in G$ has 1 as an eigenvalue in its action on $V$.  In this paper we show that certain unisingular representations can be realized as mod 2 representations of hyperelliptic Jacobians over $\Q$.  We additionally identify new unisingular representations of the symmetric and alternating groups.
\end{abstract}

\maketitle

\section{Introduction}

\subsection{Background and Motivation}

In this note we look at a specific question at the intersection of Inverse Galois Theory and the arithmetic of abelian varieties.  Our goal here is to raise the question of whether certain groups $G$ can be realized as Galois over $\Q$, where we additionally specify the degree $d$ of the generating polynomial that we seek  and, furthermore,  ask for the natural symplectic embedding 
\[
G \hookrightarrow {\rm Perm}(d) \hookrightarrow \Sp_{2 \lfloor (d-1)/2 \rfloor}(2)
\]
to be irreducible (or absolutely irreducible).

Setting some terminology, if $\rho: G \to \GL(V)$ is a finite dimensional representation of a finite group, we say that $\rho$ is \textbf{unisingular} if $\det(1-\rho(g)) =0$ for all $g \in G$.  Alternatively, we say that $\rho(G)$ is a \textbf{fixed-point} subgroup of $\GL(V)$.   Unisingular representations have been studied independently in the context of Lie Theory (see, \emph{e.g.}~\cite{zalesski}) and in number theory (see, \emph{e.g.}~\cite{cy}) and, recently, together \cite{cz}.  In the number theoretic context, certain unisingular representations can be used to describe point-counts on abelian varieties over finite fields by interpreting the elements of $G$ as suitable Frobenius elements in a certain fixed-point theorem.  

Since sums of unisingular representations are unisingular, it makes sense to focus on (and possibly classify all) irreducible unisingular representations.  This is still an enormous task, but when the vector space $V$ is defined over a finite field $k$, some progress has been made.  Two ways to approach this problem are to 1) fix $\dim V$ and vary $k$, and 2) study specific classes of groups and determine for which dimensions $\dim V$ and fields $k$  they have unisingular representations.  For example, 2-dimensional unisingular representations, and 4- and 6-dimensional unisingular \emph{symplectic} representations over finite prime fields are necessarily reducible (see \cite{katz} and \cite{sp6}), while all 3-dimensional irreducible unisingular representations were classified in \cite{gl3}.  For Lie-theoretic applications, we point the reader to \cite{gt} and \cite{zalesski} as just some of many examples of that literature.

Let us now turn to the application that concerns us here.  Let $A$ be an abelian variety of dimension $g \geq 1$ defined over a number field $K$.   If $\p \subseteq \OK$ is a non-zero prime of good reduction for $A$, then write $A_\p$ for the reduction modulo $\p$ of $A$.  Then $A_\p$ is a $g$-dimensional abelian variety defined over the residue field $\Fp := \OK/\p$. Fix a rational prime number $\ell$ and consider the following phenomenon.  Suppose $A_\p$ has a point of order $\ell$ defined over $\Fp$, for every good $\p$, but no member of the $K$-isogeny class of $A$ has a global point of order $\ell$.  In other words, suppose all but finitely many of the finite groups $A_\p(\Fp)$ have order divisible by $\ell$, but neither $A(K)$ nor any member $A'$ of the $K$-isogeny class of $A$ has a point of order $\ell$ defined over $K$.   By contrast, if \emph{any}  member of the $K$-isogeny class of $A$ had a point of prime order $\ell$ defined over $K$, then all $A'_\p$ would have a point of order $\ell$ defined over $\Fp$ for all primes $\p$ of good reduction and \emph{all} $K$-isogenous $A'$. 

The first substantial work on this setup was done by Katz in \cite{katz} where he showed that it amounts to understanding unisingular symplectic representations.  Specifically, the hypothesis that $A_\p(\Fp)$ has a point of fixed order $\ell$, for almost all primes $\p$, means that $\det (1 - \overline{\rho}_{A,\ell}(g)) = 0$ for all $g \in \Gal(\overline{K}/K)$, where 
\[
\overline{\rho}_{A,\ell}: \Gal(\overline{K}/K) \to \GSp_{2g}(\Fl)
\]  
is the mod $\ell$ representation describing the action of Galois on the $\ell$-torsion points of $A$.  Subsequently, whether or not any members of the $K$-isogeny class of $A$ have a global point of order $\ell$ is equivalent to whether or not the Jordan-H\"older series (semisimplification) of the representation $\overline{\rho}_{A,\ell}$ contains the trivial representation.  Duly motivated, we can now state our group-theoretic question as follows.

\begin{qu}
For all prime numbers $\ell$ and all positive integers $g$, can we classify the \textbf{irreducible} fixed-point subgroups $G$ of $\GSp_{2g}(\Fl)$?  
\end{qu}

Such a classification would then  give us a classification of mod $\ell$ representations attached to abelian varieties defined over a number field $K$ that have a point of order $\ell$ when reduced modulo all but finitely many $\p$, but neither they nor any member of their $K$-isogeny class has a global point of order $\ell$.

For abelian varieties of dimensions 1 and 2 (so, symplectic representations of dimensions 2 and 4, respectively), Katz showed that for any prime $\ell$, if the mod $\ell$ representation is unisingular, then its Jordan-H\"older series contains the trivial representation.  This largely amounts to a group theory argument showing that fixed-point subgroups $G$ of $\GL_2(\Fl)$ (resp.~$\GSp_4(\Fl)$) have a trivial composition factor; concretely, up to a global change of basis, the elements of $G$ fix a common line.

For abelian threefolds (\emph{i.e.} for the group $\GSp_6(\Fl)$) this is no longer the case for odd $\ell$: there exist (reducible) unisingular representations whose Jordan-H\"older series do not contain the trivial representation.  However, when $\ell=2$ all fixed-point subgroups of $\Sp_6(\F_2)$ contain the trivial representation, as shown in  \cite{2tors}. In dimensions $\geq 4$ we get our first examples of irreducible, symplectic, unisingular representations.  Namely, in \cite{2tors} we raised the example, originally communicated  to us by Serre \cite{serre_letter}, that the Steinberg representation $\St: \SL_3(\F_2) \to \Sp_8(\F_2)$ is both unisingular and absolutely irreducible.  Using this example as a starting point, in \cite{cz} we classified several new families of unisingular symplectic representations, which we will recall below.  Equipped with a unisingular symplectic representation, we can turn to arithmetic and ask whether or not we can construct an abelian variety $A$ with this specific representation, \emph{while minimizing the degree of the number field over which $A$ is defined}. 

Why do we ask for a minimal degree of definition?  Fix $\ell$ and $g$ and suppose that $G \subseteq \GSp_{2g}(\Fl)$ is a fixed-point subgroup.  Let $A/\Q$ be an abelian variety such that the mod $\ell$ representation 
\[
\overline{\rho}_{A,\ell}: \Gal(\overline{\Q}/\Q) \to \GSp_{2g}(\Fl)
\]
is surjective (such $A$ are known to exist by general principles).  Then consider the base-change to the fixed field of $G$:
\[
\xymatrix{
&\Q(A[\ell]) \ar@{-}^{\GSp_{2g}(\Fl)}[dd]  \ar@{-}_{G}[dl]\\
K:=\Q(A[\ell])^{G} \ar@{-}[dr]& \\
&\Q}
\]
This produces what we want: an abelian variety $A$, defined over a number field $K$ with mod $\ell$ image isomorphic to  $G$.  However, the degree $[K:\Q]$ is typically very large and such a construction does not necessarily give a deeper understanding of the connection between unisingular Galois representations and abelian varieties.  For example, as we will see below, we can construct a 4 dimensional abelian variety defined over $\Q$ with an irreducible unisingular mod 2 representation of order $2^4.3^3$.  By comparison, the order of $\Sp_{8}(\F_2)$ is $2^{16}.3^5.5^2.7.17$.  In general, there may be some arithmetic obstruction to constructing abelian varieties over $\Q$ with specified mod $\ell$ image (\emph{e.g.}~surjective determinant), which is why we ask more generally for a minimal field of definition, rather than asking for it over $\Q$.




When $\ell = 2$, there is a natural source of potential examples: hyperelliptic Jacobians.  If $k$ is a field of characteristic different from 2, $C_f$ the hyperelliptic curve of genus $g$ defined over $k$ by the equation
\[
y^2 = f(x),
\]
and $J_f :=\Jac(C_f)$ its Jacobian, then it is well known that the 2-torsion field $k(J_f[2])$ has Galois group isomorphic to $\Gal_k(f)$.  As above, set $d = \deg(f)$ so that $g = \lfloor \frac{d-1}{2} \rfloor$.  Studying the image of the mod 2 representation for unisingularity amounts to understanding the Galois group $\Gal_k(f)$ and its symplectic embedding 
\[
\Gal_k(f) \to {\rm Perm}(d) \to \Sp_{2g}(\F_2),
\]
where the first inclusion is basic Galois theory and the second is a standard construction in permutation group theory (see \cite[\S 2]{yelton_sym} for details).  This raises the next main question of the paper -- the realization problem.


\begin{qu}
Can irreducible fixed-point subgroups of $\Sp_{2g}(\F_2)$ be realized as Galois groups over $\Q$ via the 2-torsion subgroup of hyperelliptic Jacobians of genus $g$?
\end{qu}

If we could answer this question for a particular group $G$, then we would have solved the realization problem for $G$ over $\Q$.  This is a problem of Inverse Galois Theory where we are additionally specifying the degree of the polynomial generating the Galois group and asking for  (absolute) irreducibility of the symplectic representation of $G$. By restricting to the prime $\ell=2$, we simplify the general problem of realization in two ways.  First, when $\ell = 2$, the  symplectic similitude group and isometry group coincide:  $\Sp_{2g}(\F_2) \simeq \GSp_{2g}(\F_2)$.  This allows us to avoid surjectivity conditions on the determinant representation when attempting to construct explicit examples. Secondly, when $\ell=2$ there is a well understood mechanism for explicit computation in the 2-torsion subgroup.  For odd $\ell$, it is only for very specific examples of Jacobians that we can explicitly write down equations of the underlying curve where we can also rigorously show that the image of the mod $\ell$ representation is exactly $G$ which, as the example above indicates, is typically much smaller than the maximal image.  By comparison, showing an image is large typically amounts to computing Frobenius elements and showing that the image meets enough conjugacy classes to generate a large subgroup of $\GSp_{2g}(\Fl)$.

\subsection{Statement of Main Results}  We are now ready to describe our main results.  Let $f \in k[x]$ be an irreducible polynomial of degree $d \geq 3$ with Galois group $G$.  Let $g = \lfloor (d-1)/2 \rfloor$.  We say that the hyperelliptic Jacobian $J_f$ \textbf{unirealizes $G$ over $k$} if the composition $\varphi$ of the natural permutation representation with its symplectic embedding (see \cite{yelton_sym} for the details of this construction)
\[
\Gal_k(f) \simeq \underbrace{G \hookrightarrow S_{d} \hookrightarrow \Sp_{2g}(\F_2)}_{\varphi}
\]
is \textbf{irreducible} and \textbf{unisingular}; note that the irreducibility of the polynomial $f$ does \emph{not} guarantee the irreducibility of the representation $\varphi: {G} \to \Sp_{2g}(\F_2)$.

The smallest $g$ for which there exist any irreducible fixed-point subgroups $G \subseteq \Sp_{2g}(\F_2)$ is $g=4$ and the maximal fixed-point subgroups of $\Sp_8(\F_2)$ are ${\rm L}_3(2):2$ and $\AGL_2(3)$ \cite{cz}.  Note that ${\rm L}_3(2):2 \simeq \PGL_2(7)$ and recall from \cite[Thm.~1.7]{cz} that $\PGL_2(q)$ has an irreducible unisingular representation into $\Sp_{q+1}(\F_2)$ when, simultaneously, $q$ is odd, $4|(q+1)$, and $3|(q-1)$ .  Thus, the group $\PGL_2(7)$ is the smallest of an infinite family of fixed-point subgroups of sympectic groups.  (The two smallest groups in this family are $\PGL_2(7)$ and $\PGL_{2}(19)$.)  Our first result is on the unirealizations of these smallest examples. 

\begin{thm}
\begin{enumerate}
\item There exist infinitely many polynomials $f \in \Q[x]$ of degree 9 such that $J_f$ unirealizes $\AGL_2(3)$ over $\Q$.
\item 
The 8-dimensional, irreducible, unisingular, $\F_2$-representation of ${\rm L}_3(2):2$ (Steinberg representation) does not occur as the mod 2 representation of a genus 4 hyperelliptic Jacobian.
\end{enumerate}
\end{thm}

Even though the proof of Item (2) is not difficult -- the group ${\rm L}_3(2):2$ is not a transitive subgroup of $S_9$ or $S_{10}$ -- it raises the more interesting question of whether or not there exists any abelian fourfold over $\Q$ with this mod 2 representation.   We point the reader to Example \ref{agl(2,3)exm} below for an explicit instance of item (1) above.

In the second part of the paper we identify new classes of characteristic zero unisingular representations.  Here we focus on the symmetric and alternating groups and show that some of their Specht modules are unisingular over $\Z$.  Not all Specht modules for $S_n$ are unisingular over $\C$; for example, the standard representation $\mathcal{S}^{(n-1,1)}$ of $S_n$ is not unisingular once $n>1$, for any $n$, since the $n$-cycles have characteristic polynomial $(x^n-1)/(x-1)$, which does not vanish at $x=1$. Nevertheless, we can prove the following.

\begin{thm}\label{specht_thm}
The Specht modules $\mathcal{S}^{(n-2,1,1)}$ and $\mathcal{S}^{(n-2,2)}$ are unisingular for $S_n$.
\end{thm}

Computation suggests that many more classes of Specht modules $\mathcal{S}^\lambda$ are unisingular, but we will not investigate them further except to state that to determine precisely which ones are unisingular would be an interesting topic for further exploration.  We similarly expect that many classes of Specht modules can be proved \emph{not} to be unisingular and put forth the following conjecture, based on evidence we supply in Section \ref{not_uni}.

\begin{conjecture}
Let $n \geq 5$ be odd.  Then the class of elements of order $2n-4$ in the conjugate representation $\mathcal{S}^{(n-2,2)'} = \mathcal{S}^{(2,2,1,\dots,1)}$ does not afford 1 as an eigenvalue. Therefore the Specht module $\mathcal{S}^{(n-2,2)'}$ is not unisingular.
\end{conjecture}

Specht modules are irreducible over $\C$ and defined over $\Z$, so it makes sense to reduce them modulo any prime $\ell$, and in particular $\ell=2$. However, they might not remain irreducible upon reduction.  For example, when $n=5$, the module $\mathcal{S}^{(3,1,1)}$ is unisingular of dimension 6, and irreducible in characteristic 0.  Modulo 2, the representation is still (clearly) unisingular, but decomposes into a sum of two 1-dimensional (trivial) factors and a 4-dimensional factor.  

So this raises the question -- which of the unisingular Specht modules of Theorem \ref{specht_thm} can be realized by hyperelliptic Jacobians?  In order to achieve this, we would require
\begin{enumerate}
\item an even-dimensional, unisingular, irreducible representation $\mathcal{S}^\lambda$ of $S_n$; write $\dim \mathcal{S}^\lambda = f^\lambda = 2d$, and
\item $\mathcal{S}^\lambda$ to be irreducible modulo 2 (see \cite[Main Theorem]{james_mathas} for precise conditions on when this occurs), and
\item $f(x) \in \Q[x]$ an irreducible polynomial of degree $2d+1$ or $2d+2$ with Galois group $S_n$, and
\item the permutation representation $S_n \to {\rm Perm}(2d+1 \text{ or } 2) \to \Sp_{2d}(\F_2)$ to coincide with the reduction modulo 2 of $\mathcal{S}^\lambda$. 
\end{enumerate}
In fact, we expect that these items are \emph{never} simultaneously satisfied, and give some evidence for this expectation at the end of the paper.  If so, this would mean that while unirealizations of the Specht modules may exist over $\Q$, even for Jacobians, they cannot be realized by \emph{hyperelliptic} Jacobians.

However, this gives a potential window into determining minimal fields of definitions for unirealizing these Specht modules and other unisingular representations by hyperelliptic Jacobians.  Namely, if we consider group extensions $S_n.Q$ and field diagrams
\[
\xymatrix{
L \ar@{-}^{S_n}[d] \\
K \ar@{-}^{Q}[d] \\
\Q
}
\]
and ask for polynomials $f(x)$ with Galois group $S_n.Q$ of the appropriate degree to unirealize $S_n$ by $J_f$ over the base change to $K$. We will not pursue this question in this paper and will be content to simply determine new families of unisingular representations of the symmetric group.


\subsection{Notation and Conventions} We have already introduced much of the notation we will use in the paper.  In general, we use \textrm{Atlas} \cite{Atlas} notation for finite groups, \emph{e.g.}~${\rm L}_n(q)$ for the finite group $\PSL_n(\F_q)$.  We will make extensive use of the LMFDB \cite{lmfdb} database for explicit polynomial and/or field extensions.  In each instance we provide a link to the LMFDB to the relevant object.  Similarly we rely on the online Atlas at \href{http://brauer.maths.qmul.ac.uk/Atlas/v3/}{http://brauer.maths.qmul.ac.uk/Atlas/v3/} for explicit matrix representations of large degree.  

\subsection{Acknowledgments} We would like to thank Gunter Malle and Alexandre Zalesski for helpful conversations.  We also thank Drew Sutherland for access to a version of \textsf{Magma} allowing us to perform some of  the lengthy eigenvalue computations in Section \ref{not_uni}.

\section{Permutation Representations and Jacobians} \label{perm}

For a detailed background on the 2-torsion of Jacobians, we refer the reader to \cite[\S 1.2]{yelton_sym}; in this section we will just give a brief sketch of how one associates to a hyperelliptic curve the 2-torsion subgroup of its Jacobian.  All of this holds in far more generality, but in this paper we will work exclusively over $\Q$ for simplicity. 

Let $C_f/\Q$ be the hyperelliptic curve defined by the squarefree polynomial $f(x)$ of degree $d \geq 3$:
\[
C_f: y^2 = f(x).
\]
Then $C_f$ has genus $g = \lfloor (d-1)/2 \rfloor$ and its Jacobian $J_f := \Jac(C_f)$ is an abelian variety of dimension $g$, also defined over $\Q$.  Let $K_f$ be the splitting field over $\Q$ of $f$ and $\Gal_\Q(f)$ its Galois group; we allow for intransitive Galois groups for the moment.

If $A$ is a principally polarized abelian variety of dimension $g$ defined over $\Q$, then the 2-torsion subgroup $A[2]$ of $A$ has the structure $A[2] \simeq (\Z/2)^{2g}$ as a Galois module, yielding the  Galois representation (the mod 2 representation)
\[
\overline{\rho}_{A,2}:\Gal(\overline{\Q}/\Q) \to \GL_{2g}(\F_2).
\]
The Weil pairing is an alternating, Galois-equivariant pairing on $A[2]$, whence the image of $\overline{\rho}_{A,2}$ lies in $\Sp_{2g}(\F_2)$. 

Via divisors, $C_f$ embeds into  $J_f$ and the 2-torsion subgroup of $J_f$ is identified with the Weierstrass points of $C_f$.  This construction shows us that 
\[
\im \overline{\rho}_{J_f,2} \simeq \Gal(K_f/\Q) \subseteq S_{d} \subseteq \Sp_{2g}(\F_2).
\]
By comparing orders, we see that once $g \geq 3$ the mod 2 representation of a hyperelliptic Jacobian cannot be surjective onto $\Sp_{2g}(\F_2)$.  For the explicit details on how one goes from the natural permutation representation on the roots of $f$ to the symplectic embedding, see \cite[Prop. 2.1]{yelton_sym}.

Therefore, if we are interested in constructing interesting (\emph{e.g.}~absolutely simple) examples of abelian varieties over $\Q$   which have an even number of points modulo all but finitely many $p$, but that no member of its $\Q$-isogeny class has a rational point of order 2, then one approach would be to follow the procedure from the Introduction:
\begin{enumerate}
\item identify groups $G$ with irreducible, unisingular, symplectic ($2g$-dimensional), representations $\rho$ over $\F_2$, and 
\item determine irreducible polynomials $f \in \Q[x]$ of degree $d = 2g+1$ or $2g+2$ with $\Gal_\Q(f) \simeq G$, such that the composition 
\[
G \to S_{d} \hookrightarrow \Sp_{2g}(\F_2)
\]
of the permutation representation with the natural inclusion into the symplectic group $\Sp_{2g}(\F_2)$ coincides with $\rho$. 
\end{enumerate}
If we can find a suitable polynomial $f$ that achieves both items, then recall from the Introduction that we say $J_f$ unirealizes $G$.   One cautionary remark is in order: it does not immediately follow that if $f$ is an irreducible polynomial then the $2g$-dimensional symplectic representation of $G$ is irreducible.  For example, in \cite[Prop.~2.6]{yuri}, Zarhin shows that if $n$ is \emph{prime}, and 2 is a primitive root modulo  $n$,  and $f(x)$ is irreducible of degree $n$, then the associated mod 2 representation on the Jacobian $J_f$ is irreducible.  (If the permutation representation is, additionally, 2-transitive, then the mod 2 representation is absolutely irreducible.)  While one may be able to prove more general relationships between the irreducibility of $f$ and the irreducibility of the associated symplectic representation, we will only give \emph{ad hoc} reasoning that our examples yield irreducible mod 2 representations.




\section{Unirealization of $\AGL_2(3)$}

Our first result concerns the smallest-sized groups in the families of unisingular representations introduced in \cite{cz}.

\begin{thm}
\begin{enumerate}
\item There exist infinitely many polynomials $f \in \Q[x]$ of degree 9 such that the hyperelliptic Jacobian $J_f$ unirealizes $\AGL_2(3)$.
\item 
The 8-dimensional, irreducible, unisingular, $\F_2$-representation of ${\rm L}_3(2):2$ does not occur as the mod 2 representation of a genus 4 hyperelliptic Jacobian.
\end{enumerate}
\end{thm}

\begin{proof}
For item 1, we refer to the constructions of Malle in \cite[\S 7]{malle1}.  There he exhibits two different 2-parameter families of rational polynomials with Galois group $\AGL_2(3)$.  For infinitely many specializations of both parameters, these polynomials are irreducible of degree 9 and have Galois group $\AGL_2(3)$ over $\Q$.  The group $\AGL_2(3)$ is maximal in $S_9$ by the O'Nan-Scott Theorem, and the inclusion $\AGL_2(3) \hookrightarrow S_9 \to \Sp_8(2)$ is absolutely irreducible and unisingular by \cite[Thm.~4.1]{cz}.  Since $\Sp_8(2)$ has a {unique} (up to conjugacy) subgroup isomorphic to $S_9$ and $S_9$ has a unique (up to conjugacy) subgroup isomorphic to $\AGL_2(3)$, the specializations of Malle's polynomials give precisely this fixed-point subgroup of $\Sp_8(2)$.   By the  process outlined above, there exists a hyperelliptic Jacobian of genus 4 defined over $\Q$ unirealizing $\AGL_2(3)$.

\bigskip

For item 2, we simply observe that ${\rm L}_3(2):2 \simeq \PGL_2(7)$ is not a transitive subgroup of $S_9$ or $S_{10}$, and so is not the Galois group of an irreducible degree 9 or 10 polynomial over $\Q$.
\end{proof}

For number theoretic constructions it is often desirable to have examples ramified at very few, small, primes.  For the family
\begin{align*}
g_{a,t}(x) &= x^9 + (6a - 21)x^8 + (15a^2 - 102a + 171)x^7 + (18a^3 - 195a^2 + 648a - 675)x^6 \\
&\ + (9a^4 - 168a^3 + 903a^2 - 1890a + 1296)x^5 + (-51a^4 + 516a^3 - 1773a^2 + 2430a \\
& \ - 972)x^4  + (-a^6 + 6a^5 + 81a^4 - 604a^3 + 1350a^2 - 972a + t)x^3 + (3a^6 - 12a^5 \\
& \ - 57a^4 + 282a^3 - 324a^2 - 3t)x^2 + (-3a^6 + 6a^5 + 18a^4 - 36a^3)x + a^6,
\end{align*}
of Malle \cite{malle1} with Galois group $\AGL_2(3)$, we have
\begin{align} \label{disc_form_1}
\disc(g_{a,t}(x)) = -2^83^9t^4a^6r(a,t)^3,
\end{align}
where 
\begin{align*} \label{disc_form_2}
r(a,t) &= a^{12} - 12a^{11} + 96a^{10} - 520a^9 + 2166a^8 - 6960a^7 - 2a^6t +        17524a^6 \\& \ + 12a^5t - 34200a^5 + 48a^4t + 49329a^4 - 272a^3t -        49572a^3 + 342a^2t \\
& \ + 26244a^2 - 108at + t^2 + 108t
\end{align*}
is irreducible in $\Q[a,t]$.  By (\ref{disc_form_1}) it is only at specializations $(a,t)$ that are powers of 2 and 3 that $\disc (g_{a,t}(x))$ can be solely divisible by the primes 2 and 3.  Finding such rational specializations amounts to determining the rational points on the algebraic curves defined by 
\[
r(a,t) = \pm 2^\alpha 3^\beta. 
\]
For a fixed value of $\pm 2^\alpha3^\beta$, Riemann-Hurwitz shows that the curve has genus $>$ 1, hence by Faltings will have finitely many rational points.  At these rational points we get interesting cases of hyperelliptic Jacobians.  We are content to highlight an illustrative example.

\begin{exm} \label{agl(2,3)exm}
The polynomial 
\[
g_{1,-32}(x) = x^9 - 15x^8 + 84x^7 - 204x^6 + 150x^5 + 150x^4 - 172x^3 - 12x^2 - 15x + 1
\]
defines a number field isomorphic to the one with LMFDB label \href{https://www.lmfdb.org/NumberField/9.3.82556485632.1}{9.3.82556485632.1} with class number 1 and discriminant $-\,2^{22}\cdot 3^{9}$.  The genus 4 hyperelliptic Jacobian defined by $y^2 = g_{1,-32}(x)$ is ramified only at the primes 2 and 3, and unirealizes $\AGL_2(3)$.  

In particular, $J_f$ enjoys the properties of being defined over $\Q$, having an even number of points over $\F_p$ for all $p \geq 5$, and no member of the $\Q$-isogeny class of $J_f$ has a global 2-torsion point. 

To conclude this example, we mention that in the paper \cite{cz} we showed that the only subgroups $G$ of $\AGL_2(3)$ that act irreducibly on the space $\F_2^8$ are 
\[
{\rm A}\Gamma {\rm L}_1(9),\ \AGL_1(9), \PSU_3(2), \text{and} \ {\rm ASL}_2(3),
\]
of orders 144, 72, 72, and 216, respectively.  For example, the specialization $g_{1,1}(x)$ has Galois group ${\rm A}\Gamma {\rm L}_1(9)$, showing that the hyperelliptic Jacobian defined by $g_{1,1}(x)$ unirealizes an even smaller irreducible fixed-point subgroup of $\Sp_{8}(2)$ than $\AGL_2(3)$. 
\end{exm}

\begin{rmk}
The next-smallest group in the unisingular family of $\PGL_2(q)$ is $\PGL_2(19)$.  We point out that there are five irreducible polynomials of degree 20 with Galois group $\PGL_2(19)$ over $\Q$ currently indexed in the LMFDB:

\begin{itemize}
\item \href{https://www.lmfdb.org/NumberField/20.2.1339432011824591977742384617725412507648.1}{20.2.1339432011824591977742384617725412507648.1} 
\item \href{https://www.lmfdb.org/NumberField/20.2.2288697943298220011370049010943499874304.1}{20.2.2288697943298220011370049010943499874304.1} 
\item \href{https://www.lmfdb.org/NumberField/20.0.14595833152648338774900083206236247150592.1}{20.0.14595833152648338774900083206236247150592.1} 
\item \href{https://www.lmfdb.org/NumberField/20.0.180218116213002899035577493449527785489956864.1}{20.0.180218116213002899035577493449527785489956864.1} 
\item \href{https://www.lmfdb.org/NumberField/20.2.206007596521214410095208558252435839890349094339.1}{20.2.206007596521214410095208558252435839890349094339.1}
\end{itemize}

\noindent While the 20-dimensional $\F_2$-representation of $\PGL_2(19)$ was shown to be unisingular in \cite[Thm.~1.7]{cz}, these polynomials define 9-dimensional hyperelliptic Jacobians.  Hence, these symplectic embeddings are $\PGL_2(19) \hookrightarrow S_{18} \to \Sp_{18}(\F_2)$.  The group $\PGL_2(19)$ has three 2-modular irreducible representations defined over $\F_2$, of degrees 1, 18, and 20.  The 18-dimensional representation is not unisingular since the characteristic polynomial of the elements of order 19 is cyclotomic of degree 18,  which does not have 1 as a root in $\F_2$.  

We would therefore require a degree 21 or 22 polynomial with Galois group $\PGL_2(19)$ in order to unirealize $\PGL_2(19)$.  However, $\PGL_2(19)$ is not a transitive subgroup of $S_{21}$ or $S_{22}$.  Thus, $\PGL_2(19)$ cannot be unirealized by a hyperelliptic Jacobian over $\Q$.   This illustrates some of the difficulties involved with constructing hyperelliptic unirealizations.  
\end{rmk}

\section{Unisingular Representations of the Symmetric Group}

Now we turn to the symmetric and alternating groups and prove that certain families of Specht modules are always unisingular, over any field.  While this may be well known to experts, we were unable to find this specific result in the literature.  

Let $n$ be a positive integer, $S_n$ the symmetric group on $n$ letters, $\lambda$ a partition of $n$, and $\mathcal{S}^\lambda$ the associated Specht module for $S_n$.  Since $\mathcal{S}^\lambda$ is defined over $\Z$, it makes sense to reduce $\mathcal{S}^\lambda$ modulo $p$ and consider it as a (possibly reducible) representation of $S_n$ over $\F_p$. 

We now focus on two partitions of $n$ in particular: $(n-2,1,1)$ and $(n-2,2)$.  It is well known that 
\[
\dim \mathcal{S}^{(n-2,2)} = n(n-3)/2 \qquad \text{and} \qquad \dim \mathcal{S}^{(n-2,1,1)} = \dim \mathcal{S}^{(n-2,2)} + 1.
\]
Our main theorem, which we quote from the Introduction, is on the unisingularity of these Specht modules.

\begin{thm}
The Specht modules $\mathcal{S}^{(n-2,1,1)}$ and $\mathcal{S}^{(n-2,2)}$ are unisingular for $S_n$.
\end{thm}

If $n$ is odd, then we conjecture below that only one class of elements fails to have 1 as an eigenvalue in the conjugate Specht module $\mathcal{S}^{(n-2,2)'} = \mathcal{S}^{(n-2,2)} \otimes \mathcal{S}^{(1,\dots,1)}$.  Moreover, restricting these representations to $A_n$ yields irreducible, unisingular representations since the class in question consists of odd permutations. 


In the proof below we use the standard notions of tableaux, tabloids, and polytabloids, as developed comprehensively in \cite{sagan}.  We adopt their notation as well: if $t$ is a tableau of shape $\lambda$, write ${\lbrace t \rbrace}$ and $e_t$ for the associated tabloid and polytabloid, respectively.  We remind the reader that a \emph{derangement} of a set of cardinality $m \geq 2$ is a permutation of that set with no fixed points. We also remind the reader that if $e_t$ is a polytabloid and $\pi \in S_n$, then $\pi e_{t} = e_{\pi t}$. 

Before getting into the details, we sketch the idea of the proof, which is simple:  for each $\sigma \in S_n$ we identify an explicit element of $\mathcal{S}^\lambda$  fixed by $\sigma$.  For many classes of elements, this is straightforward since if $\sigma$ is a derangement of $m \leq n-3$ or $n-4$ elements, then put the fixed elements into the unique column of a tableau of shape $(n-2,1,1)$, or into the first two columns of a tableau of shape $(n-2,2)$.  For derangements of larger cardinalities we construct a fixed element as follows.  

Let $t$ be a tableau of shape $\lambda$ and let $\sigma \in S_n$ be of order $k$.  Define 
\[
E_{\sigma,t} = \sum_{j=0}^{k-1} e_{\sigma^jt}.
\]
It is clear that $\sigma$ fixes $E_{\sigma,t}$ since $\sigma e_t = e_{\sigma t}$ and the action is linear, but it is possible that $E_{\sigma,t}$ may be the zero polytabloid.  For example, in the Specht module $\mathcal{S}^{(n-1,1)}$, if $\sigma$ is an $n$-cycle and $t$ is any tableau, then $E_{\sigma,t} = 0$; consider $\sigma = (12\cdots n)$ and $t$ the tableau
\[
\begin{tabular}{cccc}
$1$ & $2$ & $\cdots$ & $n-1$ \\
$n$  
\end{tabular}.
\]
Then 
\[
\sigma^jt =
\begin{tabular}{cccc}
$j+1$ & $j+2$ & $\cdots$ & $j-1$ \\
$j$  
\end{tabular},
\]
so
\begin{align*}
e_{\sigma^jt} = 
\begin{tabular}{cccc}
\hline
$j+1$ & $j+2$ & $\cdots$ & $j-1$ \\
\hline
\underline{ \ \ $j$ \ \  }  
\end{tabular} - 
\begin{tabular}{cccc}
\hline
$j$ & $j+2$ & $\cdots$ & $j-1$ \\
\hline
\underline{ \ \ $j+1$ \ \  }
\end{tabular} := \mathbf{j} - (\mathbf{j+1}).
\end{align*}
Therefore 
\[
E_{\sigma,t} = \sum_{j=0}^{n-1} e_{\sigma^jt} = \mathbf{n} - \mathbf{1} + \sum_{j=1}^{n-1} \mathbf{j} - (\mathbf{j+1}) = 0,
\]
and so is not a candidate for a nontrivial fixed vector.  Much of our proof below will amount to showing that $E_{\sigma,t} \ne 0$ for various classes of $\sigma$.  

\medskip

\noindent \textbf{Notation.} For ease of exposition, if $t$ is the tableau
\[
t:=\begin{tabular}{cccc}
$a$ & $d$ & $\cdots$ & $e$ \\
$b$  \\
$c$
\end{tabular},
\]
then we write the associated tabloid $\lbrace t \rbrace$ as $\left\{\begin{smallmatrix} a \\ - \\ b \\ - \\ c \end{smallmatrix}\right\}$ and polytabloid $e_t$ as $\left[\begin{smallmatrix} a \\ - \\ b \\ - \\ c \end{smallmatrix}\right]$, so that
\[
\left[\begin{smallmatrix} a \\ - \\ b \\ - \\ c \end{smallmatrix}\right] = \left\{\begin{smallmatrix} a \\ - \\ b \\ - \\ c \end{smallmatrix}\right\} + 
\left\{\begin{smallmatrix} b \\ - \\ c \\ - \\ a \end{smallmatrix}\right\} +
\left\{\begin{smallmatrix} c \\ - \\ a \\ - \\ b \end{smallmatrix}\right\} -
\left\{\begin{smallmatrix} a \\ - \\ c \\ - \\ b \end{smallmatrix}\right\} -
\left\{\begin{smallmatrix} c \\ - \\ b \\ - \\ a \end{smallmatrix}\right\} -
\left\{\begin{smallmatrix} b \\ - \\ a \\ - \\ c \end{smallmatrix}\right\}.
\]
Similarly, if $s$ is the tableau
\[
s:=\begin{tabular}{cccc}
$a$ & $b$ & $\cdots$ & $e$ \\
$c$ & $d$  
\end{tabular},
\]
then we write $\lbrace s \rbrace = \left\{\begin{smallmatrix} a & b\\ c & d \end{smallmatrix}\right\}$ and $e_s = \left[\begin{smallmatrix} a & b \\ c & d \end{smallmatrix}\right] = \left\{\begin{smallmatrix} a & b\\ c & d \end{smallmatrix}\right\} - \left\{\begin{smallmatrix} c & b\\ a & d \end{smallmatrix}\right\} - \left\{\begin{smallmatrix} a & d\\ c & b \end{smallmatrix}\right\} + \left\{\begin{smallmatrix} c & d\\ a & b \end{smallmatrix}\right\}$.

\begin{proof}[Proof of Theorem \ref{specht_thm}, $\lambda = (n-2,1,1)$]
Let $X = \lbrace 1,\dots,n \rbrace$ and identify ${\rm Perm}_X$ with $S_n$.  If $Y \subseteq X$ has cardinality $m$, then we may naturally identify ${\rm Perm}_Y$ with a subgroup of ${\rm Perm}_X$, isomorphic to $S_m$.   Let $Y  = \lbrace y_1,\dots,y_m \rbrace \subseteq X$ have cardinality $m \leq n-3$. Write $X \setminus Y = \lbrace x_1,\dots,x_{n-m} \rbrace$.   Let $\sigma \in {\rm Perm}_Y$ and consider the tableau
\[
t:=\begin{tabular}{ccccccc}
$x_1$ & $x_4$ & $\cdots$ & $x_{n-m}$ & $y_1$ & $\cdots$ & $y_{m}$ \\
$x_2$  \\
$x_3$
\end{tabular}.
\]
Then $\sigma e_t = e_{\sigma t} = e_t$, with the latter equality due to $\sigma$ fixing the $x_i$.  Hence, permutations fixing 3 or more letters have fixed vectors in the Specht module $\mathcal{S}^{(n-2,1,1)}$. 

We break the rest of the proof up according to whether $\sigma$ is a derangement on $n-2$, $n-1$, or $n$ letters.  For notational ease, if $\sigma$ is a derangement of the subset $Y = \lbrace y_1,\dots, y_m \rbrace$ then we relabel its elements and take $Y = \lbrace 1,2,\dots,m \rbrace$.  For each class of derangements $\sigma$, we will define a tableau $t$ and then show that the sum $E_{\sigma,t}$ of polytabloids is nonzero.  To do this, we use Garnir calculus to write the constituents of $E_{\sigma, t}$ as sums of standard polytabloids.  Since the standard polytabloids form a basis for the Specht module, we will, in each case,  conclude that $E_{\sigma,t} \ne 0$ by a dimension count. We give explicit details in one case, and then will be content to simply list examples of  fixed vectors in a table, since the corresponding details are nearly identical in each case.

\bigskip

\noindent \textbf{\fbox{Case 1: $|Y| = n-2$}}  Let $\sigma \in {\rm Perm}_Y$ and write $\sigma$ as a product of disjoint cycles.  If the smallest cycle is a transposition, write $\sigma = (ab) \pi$, where $\pi$ is a derangement of the $n-4$ letters $c,\dots,d$.  We take $1$ and $2$ to be fixed by $\sigma$.  Let 
\[
t:=\begin{tabular}{ccccc}
$1$ & $c$ & $\cdots$ & $d$ &  $b$ \\
$2$  \\
$a$
\end{tabular}.
\]
Let $m$ be the order of $\sigma$ (which is even).  Then it is straightforward to work out that 
\[
E_{\sigma,t} = \sum_{j=0}^{m-1} e_{\sigma^j t} = \frac{m}{2} \left( \left[ \begin{smallmatrix} 1 \\ - \\ 2 \\ - \\ a \end{smallmatrix} \right] + \left[ \begin{smallmatrix} 1 \\ - \\ 2 \\ - \\ b \end{smallmatrix} \right] \right),
\]
which is nonzero because it is an independent sum of standard (basis) polytabloids.

\medskip

If the smallest cycle in the decomposition of $\sigma$ is a 3-cycle, then write $\sigma  = (abc) \pi$, where $\pi$ is a derangement of the $n-5$ letters $d,\dots,e$.  Let $f$ and $g$ be fixed by $\sigma$.  Let $t$ be the tableau
\[
t:=\begin{tabular}{cccccc}
$a$ & $d$ & $\cdots$ & $e$ &  $f$ & $g$ \\
$b$  \\
$c$
\end{tabular}.
\]
Then $\sigma e_t = e_{\sigma t} = e_{t}$.  

\medskip

Next, suppose that the smallest factor in the cycle decomposition of $\sigma$ is a $k$-cycle, with $4 \leq k \leq \lfloor \frac{n+1}{2} \rfloor$.  Without loss of generality, we take the letters 1 and 2 to be fixed and $\sigma = (3 \cdots k+2)\pi$, where $\pi$ is a derangement of the remaining letters, and 
\[
t:=\begin{tabular}{ccccccc}
$1$ & $4$ & $\cdots$ & $k+2$ &  $k+3$ & $\cdots$  & $n$ \\
$2$  \\
$3$
\end{tabular}.
\]
We let $m$ denote the order of $\sigma$.  It is routine to verify that 
\[
\sigma^j t  =\begin{tabular}{ccccccc}
$1$ & $j+4$ & $\cdots$ & $j+2$ &  $\pi^j(k+3)$ & $\cdots$  & $\pi^j(n)$ \\
$2$  \\
$j+3$
\end{tabular}.
\]
The polytabloid $e_{\sigma^j t}$ has no column descents and is row-equivalent to the standard polytabloid \begin{tiny}$\left[ \begin{smallmatrix} 1 \\ - \\ 2 \\ - \\ j+3 \end{smallmatrix} \right]$\end{tiny}.  Therefore, 
\[
E_{\sigma,t}  = \left(\frac{m}{k} \right) \sum_{j=0}^{k-1} \left[ \begin{smallmatrix} 1 \\ - \\ 2 \\ - \\ j+3 \end{smallmatrix} \right],
\]
as a sum of basis polytabloids is non-zero and, by construction, fixed by $\sigma$.

\medskip

Finally, if $\sigma$ is an $(n-2)$-cycle then by relabeling the elements it suffices to assume $\sigma  = (1\cdots n-2)$. Let $t$ be the tableau
\[
t:=\begin{tabular}{cccccc}
$1$ & $2$ & $\cdots$ & $n-2$ \\
$n-1$  \\
$n$
\end{tabular}.
\]
Then for $j=0,\dots n-3$ we have
\[
\sigma^jt=\begin{tabular}{ccccccc}
$j+1$ & $j+2$ & $\cdots$ & $n-2$ & 1 & $\cdots$  & $j$\\
$n-1$  \\
$n$
\end{tabular}.
\]
By Garnir calculus, and by replacing polytabloids with their row- and column-equivalents, it is straightforward to verify that for $j = 1,\dots,n-3$ we have
\[
e_{\sigma^j t} = \left[ \begin{smallmatrix} j+1 \\ - \\ n-1 \\ - \\ n \end{smallmatrix} \right] = \left[ \begin{smallmatrix} 1 \\ - \\ n-1 \\ - \\ n \end{smallmatrix} \right] - \left[ \begin{smallmatrix} 1 \\ - \\ j+1 \\ - \\ n \end{smallmatrix} \right] + \left[ \begin{smallmatrix} 1 \\ - \\ j+1 \\ - \\ n-1 \end{smallmatrix} \right],
\]
expressing $e_{\sigma^j t}$ as a sum of standard polytabloids. Summing over $j$, we obtain 
\begin{align*}
E_{\sigma, t} &= \sum_{j=0}^{n-3} e_{\sigma^j t} \\
&= \left[ \begin{smallmatrix} 1 \\ - \\ n-1 \\ - \\ n \end{smallmatrix} \right] + \sum_{j=1}^{n-3} \left[ \begin{smallmatrix} 1 \\ - \\ n-1 \\ - \\ n \end{smallmatrix} \right]- \left[ \begin{smallmatrix} 1 \\ - \\ j+1 \\ - \\ n \end{smallmatrix} \right] + \left[ \begin{smallmatrix} 1 \\ - \\ j+1 \\ - \\ n-1 \end{smallmatrix} \right] \\
&= (n-2) \left[ \begin{smallmatrix} 1 \\ - \\ n-1 \\ - \\ n \end{smallmatrix} \right] - \sum_{j=1}^{n-3} \left[ \begin{smallmatrix} 1 \\ - \\ j+1 \\ - \\ n \end{smallmatrix} \right]+ \sum_{j=1}^{n-3}\left[ \begin{smallmatrix} 1 \\ - \\ j+1 \\ - \\ n-1 \end{smallmatrix} \right].
\end{align*} 
This is a sum of $3n-8$ basis vectors in a space of dimension $(n-1)(n-2)/2$.  Once $n \geq 6$ this is an independent sum by comparing dimensions.  When $n=5$, the representation is easily verified to be unisingular anyway.

\medskip

\noindent \textbf{\fbox{Cases 2 \& 3: $|Y| \in \lbrace n-1,n \rbrace$}} In the following table we give examples of fixed vectors for each class of permutation $\sigma$.  The cases themselves break up into four subcases, depending on the factorization of $\sigma$ into disjoint permutations.  The four subcases depend on the size of the smallest cycle type occurring in the factorization.  Therefore, we consider when $\sigma$ is a single cycle; when $\sigma$ factors as a transposition and a derangement of the remaining letters; a 3-cycle and a derangement; and a $k$-cycle and a derangement, where $k \geq 4$.  For ease of exposition, we denote by 
\[
(a\cdots b)[c \dots d],
\] 
the product of a cycle $(a\cdots b)$ and a derangement of the remaining letters $[c \cdots d]$, which may or may not be a cycle.  We denote by $m$ the order of $\sigma \in S_n$.  In the table below, we note that each $E_{\sigma,t}$ is written as a sum of standard (basis) polytabloids of dimension $\leq \dim \mathcal{S}^{(n-2,1,1)}$, and is therefore nonzero. We label any fixed elements in {\color{red} red}.  This completes the proof, since for each class of permutation it suffices (by relabeling elements) to consider a single element and show that it has a non-zero fixed vector in the vector space $\mathcal{S}^{(n-2,1,1)}$. 

\medskip

\begin{tiny}
\begin{center}
\begin{tabular}{|c|c|c|c|}
\hline
$|Y|$ & $\sigma$ & Tableau $t$ & Fixed Vector \\
\hline
$n-1$ & $(12\cdots (n-1))$ &\begin{tiny}\begin{tabular}{cccc}
$1$ & $2$ & $\cdots$ & $n-2$ \\
$n-1$  \\
${\color{red} n}$
\end{tabular}\end{tiny} & \begin{tiny} $E_{\sigma,t} = -\sum_{j=2}^{n-2}   \left[ \begin{smallmatrix} 1 \\ - \\ j \\ - \\ j+1 \end{smallmatrix} \right]$ \end{tiny} \\
\hline 
$n-1$ & $((n-1)n)[2\cdots(n-3)]$, &  \begin{tiny}$\begin{tabular}{cccccc}
${\color{red}1}$ & $2$ & $\cdots$ & $n-3$ &  $n-1$   \\
$n-2$  \\
$n$
\end{tabular}$\end{tiny}& $E_{\sigma,t} =  \sum_{j = 0 \atop j\text{ even}}^{m-1} \left[ \begin{smallmatrix} 1 \\ - \\ \sigma^j(n-3) \\ - \\ n \end{smallmatrix} \right]$\\
&&&  $+  \sum_{j = 0 \atop  j\text{ odd}}^{m-1} \left[ \begin{smallmatrix} 1 \\ - \\ \sigma^j(n-3) \\ - \\ n-1 \end{smallmatrix} \right]$  \\
\hline
$n-1$ & $(abc)[d\cdots e]$ & \begin{tiny}$\begin{tabular}{cccccc}
$a$ & $d$ & $\cdots$ & $e$ &  ${\color{red}f}$   \\
$b$  \\
$c$
\end{tabular}$\end{tiny}& $e_t$ \\
\hline
$n-1$ & $(23\cdots k+1 ) [(k+2) \cdots n]$ & \begin{tiny}\begin{tabular}{cccccc}
${\color{red}1}$ & $2$ & $\cdots$ & $k-1$& $\cdots$ & $n$ \\
$k$  \\
$k+1$
\end{tabular}\end{tiny}  & $E_{\sigma,t} = \left(\frac{m}{k} \right) \left(\left[ \begin{smallmatrix} 1 \\ - \\ k \\ - \\ k+1 \end{smallmatrix} \right] - \left[ \begin{smallmatrix} 1 \\ - \\ 2 \\ - \\ k \end{smallmatrix} \right] + \sum_{j=2}^{k-1} \left[ \begin{smallmatrix} 1 \\ - \\ j \\ - \\ j+1 \end{smallmatrix} \right]\right)$\\
\hline
$n$ & $(12\cdots n)$ & \begin{tiny}\begin{tabular}{cccc}
$1$ & $2$ & $\cdots$ & $n-2$ \\
$n-1$  \\
$n$
\end{tabular}\end{tiny} & \begin{tiny} $E_{\sigma,t} = 2 \sum_{j=2}^{n-1} \left[ \begin{smallmatrix} 1 \\ - \\ j \\ - \\ j+1 \end{smallmatrix} \right]  - \sum_{j=2}^{n-2} \left[ \begin{smallmatrix} 1 \\ - \\ j \\ - \\ j+2 \end{smallmatrix} \right]$ \end{tiny}\\
\hline 
$n$ & $(1n)[23\cdots(n-1)]$ &\begin{tiny}\begin{tabular}{cccc}
$1$ & $2$ & $\cdots$ & $n-2$ \\
$n-1$  \\
$n$
\end{tabular}\end{tiny} & $E_{\sigma, t} = \sum_{j=0}^{m-1} (-1)^j  \left[ \begin{smallmatrix} 1 \\ - \\ \sigma^j(n-1) \\ - \\ n \end{smallmatrix} \right]$ \\
\hline
$n$ & $(abc)[d \cdots e]$ & \begin{tiny}$\begin{tabular}{cccccc}
$a$ & $d$ & $\cdots$ & $e$    \\
$b$  \\
$c$
\end{tabular}$\end{tiny}& $e_t$ \\
\hline
$n$ & $(1\cdots k)[(k+1) \cdots n]$ & \begin{tiny}$\begin{tabular}{cccccccc}
$1$  & $\cdots$ & $k-2$ & $k+1$ & $\cdots$ & $n$    \\
$k-1$  \\
$k$
\end{tabular}$\end{tiny}  & $E_{\sigma, t} = \left(\frac{m}{k} \right) \left(\left[ \begin{smallmatrix} 1 \\ - \\ k-1 \\ - \\ k \end{smallmatrix} \right] + \left[ \begin{smallmatrix} 1 \\ - \\ 2 \\ - \\ 3 \end{smallmatrix} \right] + \left[ \begin{smallmatrix} 1 \\ - \\ 2 \\ - \\ k \end{smallmatrix} \right] \right)$  \\
&&& $+ \frac{m}{k} \sum_{j=3}^{k-1} \left[ \begin{smallmatrix} 1 \\ - \\ j \\ - \\ j+1 \end{smallmatrix} \right] - \left[ \begin{smallmatrix} 1 \\ - \\ j-1 \\ - \\ j+1 \end{smallmatrix} \right] + \left[ \begin{smallmatrix} 1 \\ - \\ j-1 \\ - \\ j \end{smallmatrix} \right]$\\
 \hline
\end{tabular}
\end{center}
\end{tiny}
\end{proof}

Next we turn to the case $\lambda = (n-2,2)$. Rather than work out explicit forms for the fixed polytabloids  like in the proof of Part 1, we will only sketch the proof since the details are nearly identical.  

\bigskip

The shape of the tableaux themselves indicate how to find the fixed polyabloids:
\[
\begin{tabular}{cccccccc}
$\square \square$ & $\underbrace{\square \square \square \cdots \square}_{n-4}$     \\
$\square \square$
\end{tabular}.
\]
If $\sigma \in S_n$ has at least four fixed points, then place four of them in the the $2 \times 2$ block; the associated polytabloid is then fixed by $\sigma$.  If $\sigma$ has 1, 2, or 3 fixed points, then relabel the elements so that the sets $\lbrace 1 \rbrace$, $\lbrace 1,2 \rbrace$, or $\lbrace 1,2,3 \rbrace$ are fixed, respectively.  In those cases, fill the $2 \times 2$ block of a tableau as follows, respectively:
\[
\begin{tabular}{cccccccc}
${\color{red} 1}$ & $*$  \\
$*$ & $*$
\end{tabular}, \qquad 
\begin{tabular}{cccccccc}
${\color{red} 1}$ & {\color{red} 2}  \\
$*$ & $*$
\end{tabular}, \qquad 
\begin{tabular}{cccccccc}
${\color{red} 1}$ & {\color{red} 2}  \\
${\color{red} 3}$ & $*$
\end{tabular}.
\]
Depending on the factorization of $\sigma$ into disjoint cycles, the remainder of the entries in the tableau $t$ can be filled so that the element $E_{\sigma_,t}$ (which is fixed by $\sigma$) can be expressed as a non-zero sum of basis polytabloids.

If $\sigma$ fixes no element of $\lbrace 1,\cdots,n \rbrace$, then we can again proceed as in the proof of Part 1 and work on a case-by-case basis, depending on the factorization of $\sigma$.  In each of the following cases:
\begin{itemize}
\item $\sigma$ is an $n$-cycle,
\item the smallest factor of $\sigma$ is a 2-cycle,
\item the smallest factor of $\sigma$ is a 3-cycle,
\item the smallest factor of $\sigma$ is a 4-cycle,
\item the smallest factor of $\sigma$ is a $k$-cycle, $5 \leq k \leq \lfloor \frac{n+1}{2} \rfloor$, 
\end{itemize}
it follows by nearly identical arguments as above that we can produce a tableau $t$ such that the element $E_{\sigma, t}$, which is fixed by $\sigma$, is nonzero.  We give the details of a sample case to conclude this sketch.

Let $n$ be even.  Suppose $\sigma$ factors as a transposition and an $(n-2)$-cycle.  Then the order of $\sigma$ is $n-2$. Relabel the elements so that $\sigma = (12)(3 \cdots n)$ and let 
\[
t = \begin{tabular}{cccccccc}
1 & 2 & 3 & $\cdots$ &$n-2$ \\
$n-1$ & $n$
\end{tabular}.
\]
Then
\[
\sigma^jt = 
\begin{cases}
\begin{tabular}{cccccccc}
1 & 2 & $j+3$ & $\cdots$ &$n$&3 & $\cdots$ & $j$ \\
$j+1$ & $j+2$ 
\end{tabular} & \text{if $2 \leq j \leq n-4$ is even} \\[10pt] 
\begin{tabular}{cccccccc}
2 & 1 & $j+3$ & $\cdots$ &$n$&3 & $\cdots$ & $j$ \\
$j+1$ & $j+2$ 
\end{tabular} & \text{if $3 \leq j \leq n-3$ is odd.}
\end{cases}
\]
If $j$ is even then $e_{\sigma^jt}$ is row-equivalent to the standard polytabloid \begin{tiny}$\left[ \begin{smallmatrix} 1 & 2 \\ j+1 & j+2 \end{smallmatrix} \right]$\end{tiny}, while if $j$ is odd then $e_{\sigma^jt}$ is row-equivalent to the polytabloid \begin{tiny}$\left[ \begin{smallmatrix} 2 & 1 \\ j+1 & j+2 \end{smallmatrix} \right]$\end{tiny}. Using Garnir calculus, we can rewrite this last polytabloid as a linear combination of basis polytabloids:
\[
\left[ \begin{smallmatrix} 2 & 1 \\ j+1 & j+2 \end{smallmatrix} \right] = \left[ \begin{smallmatrix} 1 & 2 \\ j+1 & j+2 \end{smallmatrix} \right] - \left[ \begin{smallmatrix} 1 & 3 \\ 2 & j+2 \end{smallmatrix} \right] + \left[ \begin{smallmatrix} 1 & 3 \\ 2 & j+1 \end{smallmatrix} \right].
\]
Then $E_{\sigma,t}$ has the following expression as a sum of basis elements:
\[
E_{\sigma,t} = \left[ \begin{smallmatrix} 1 & 2 \\ 3 & n \end{smallmatrix} \right] + \sum_{j=3}^{n-1}  \left[ \begin{smallmatrix} 1 & 2 \\ j & j+1 \end{smallmatrix} \right] + \sum_{j=2}^{(n-2)/2}  \left (\left[ \begin{smallmatrix} 1 & 2 \\ 2 & 2j \end{smallmatrix} \right] -  \left[ \begin{smallmatrix} 1 & 2 \\ 2 & 2j+1 \end{smallmatrix} \right] \right),
\]
which is nonzero by a dimension count.  

\bigskip

Now that we have identified new classes of unisingular representations, we can ask about the associated unirealization problem.  If $\mathcal{S}^\lambda$ is a Specht module for $S_n$ or $A_n$, then we can consider it as a representation over $\F_2$ simply by reducing the matrix entries modulo 2.  And if $\mathcal{S}^\lambda$ is additionally unisingular, then the unisingularity is preserved by reduction.

By \cite[Main Theorem]{james_mathas} the Specht module $\mathcal{S}^{(n-2,1,1)}$ is \emph{never} irreducible mod 2 since both $(n-2,1,1)$ and the conjugate partition $(3,1,\dots,1)$ are 2-singular.  In fact, the representation $\mathcal{S}^{(n-2,1,1)}$ contains a copy of the trivial representation in its decomposition modulo 2.  On the other hand, the representation $\mathcal{S}^{(n-2,2)}$ is irreducible for certain values of $n$; for $5 \leq n \leq 13$ it is irreducible precisely when $n=7,11$.  

Since $S_7$ is not a transitive subgroup of $S_{15}$ or $S_{16}$, this representation of $S_7$, which is 14-dimensional, cannot be unirealized over $\Q$ by a hyperelliptic Jacobian; similarly for $S_{11}$.  But all of this raises two problems for further investigation:

\begin{enumerate}
\item Determine which Specht modules are unisingular, either for $S_n$ or $A_n$.
\item Given a unisingular $\mathcal{S}^\lambda$ for $S_n$ (or $A_n$) of dimension $f^\lambda = 2d$ which remains irreducible mod 2, determine whether or not there are transitive groups $T \subseteq S_{2d+1}$ or $S_{2d+2}$ with quotient $S_n$ (or $A_n$).  This could give a bound on the minimal degree needed for a hyperelliptic unirealization of the representation  $\mathcal{S}^\lambda$.
\end{enumerate}

\section{A Conjecture on Non-Unisingularity} \label{not_uni}

If $n$ is odd then we conjecture that a certain Specht module is never unisingular for $S_n$, but is unisingular  for $A_n$.  We first state the conjecture and then provide evidence.

\begin{conjecture}\label{nonconj}
Let $n \geq 5$ be odd.  Then the Specht module $\mathcal{S}^{(n-2,2)'}$ does not afford 1 on the class of elements of that factor as a transposition and an $(n-2)$-cycle.  Otherwise, every element has 1 as an eigenvalue. 
\end{conjecture}

When $n$ is odd the class of permutations of Conjecture \ref{nonconj} is odd, hence does not meet $A_n$.  A consequence of our conjecture is that the restriction $\mathcal{S}^{(n-2,2)'} \downarrow_{A_n}$ is unisingular.  However, because $(n-2,2)' \ne (n-2,2)$,
\[
\mathcal{S}^{(n-2,2)'} \downarrow_{A_n} = \mathcal{S}^{(n-2,2)} \downarrow_{A_n}
\]
which we already knew to be unisingular by Theorem \ref{specht_thm}.

Continuing, let $C_{(n-2,2)}$ be the class of elements of $S_n$ that decompose as a product of an $(n-2)$-cycle and a transposition.  To show that this class does not afford 1 as an eigenvalue in the representation $\mathcal{S}^{(n-2,2)'}$ we will 
\begin{enumerate}
\item compute the multiplicity of $(x-1)$ in the characteristic polynomial of the permutation module $\mathcal{M}^{(n-2,2)'}$ on $C_{(n-2,2)}$, and
\item compare this to the multiplicities of $(x-1)$ in the simple factors of $\mathcal{M}^{(n-2,2)'}$ on $C_{(n-2,2)}$; since $\mathcal{S}^{(n-2,2)'}$ appears once as a factor of $\mathcal{M}^{(n-2,2)'}$, we can just compare powers of $(x-1)$ and conclude that $\mathcal{S}^{(n-2,2)'}$ is not unisingular.
\end{enumerate}
An advantage to getting the eigenvalue-1 multiplicity in this way is that the permutation modules are  easy to generate in \textsf{Magma}.

Recall that the permutation module $\mathcal{M}^{(n)}$ and the Specht module $\mathcal{S}^{(n)}$ are isomorphic to the trivial module.  Then, for example, 
\begin{align*}
\mathcal{M}^{(n)} &= \mathcal{S}^{(n)} \\
\mathcal{M}^{(n-1,1)} &= \mathcal{S}^{(n-1,1)} \oplus \mathcal{S}^{(n)}\\
\mathcal{M}^{(n-2,2)} &= \mathcal{S}^{(n-2,2)} \oplus \mathcal{S}^{(n-1,1)} \oplus \mathcal{S}^{(n)}\\
\mathcal{M}^{(n-2,1,1)} &= \mathcal{S}^{(n-2,1,1)} \oplus \mathcal{S}^{(n-2,2)} \oplus 2\mathcal{S}^{(n-1,1)} \oplus \mathcal{S}^{(n)}.
\end{align*}
Now we work through the simplest case where $n=5$ to give a sense of what is involved, though this process generalizes naturally to larger $n$.  Our goal is in this case is to show that the characteristic polynomial of the class $C_{(n-2,2)}$ in the representation $\mathcal{S}^{(2,2,1)}$ is not divisible by $(x-1)$.

The trivial module $\mathcal{M}^{(5)} = \mathcal{S}^{(5)}$ has one power of $(x-1)$ on the class $C_{(n-2,2)}$.  Next, there is a unique transitive action of $S_5$ on 5 points, decomposing into $\mathcal{S}^{(4,1)} \oplus \mathcal{S}^{(5)}$ as above.  By computing with the character of this representation, we conclude that $\mathcal{S}^{(4,1)}$ has one power of $(x-1)$ on the class $C_{(n-2,2)}$; this is basic representation theory, and also easily implementable in \textsf{Magma}.

To construct the module $\mathcal{M}^{(3,2)}$, we search the \texttt{TransitiveGroups} database for transitive actions of $S_5$ on 10 points.  There are exactly two such actions.  Setting \texttt{T:=TransitiveGroups(10)}, the two transitive groups isomorphic to $S_5$ are \texttt{T[12]} and \texttt{T[13]}.

The group \texttt{T[13]} is the one we want, which we verify by decomposing the representation in \textsf{Magma} into irreducibles and comparing dimensions of the simple factors agains those of $\mathcal{S}^{(3,2)}$, $\mathcal{S}^{(4,1)}$, and $\mathcal{S}^{(5)}$.     The following code produces matrix generators of the permutation representation, say $\mathtt{a}$ and $\mathtt{b}$ (this is the result of the command \texttt{M:=Maximal;}), and then decomposes the permutation representation into irreducibles.

\medskip

\begin{tabular}{l}
\texttt{T:=TransitiveGroups(10);} \\
\texttt{M := PermutationModule(T[13], Rationals());}\\
\texttt{M:Maximal;} \\[5pt]

\texttt{A:=MatrixAlgebra<Rationals(), 10 | a,b>;} \\
\texttt{MM:=RModule(A);} \\
\texttt{B:=CompositionFactors(MM);} \\
\texttt{B;}
\end{tabular}

\medskip

\noindent The characteristic polynomial of the class $C_{(n-2,2)}$ on $\mathcal{M}^{(3,2)}$ is exactly divisible by $(x-1)^3$, thus we conclude that $\mathcal{S}^{(3,2)}$ has exactly one power of $(x-1)$ on the same class.

Iterating this procedure, we find that there are exactly three transitive permutation representations of $S_5$ on 20 points; if \texttt{TT:=TransitiveGroups(20)}, then they are \texttt{TT[30]}, \texttt{TT[32]}, and \texttt{TT[35]}.  By decomposing each module, comparing to the known decomposition of $\mathcal{M}^{(3,1,1)}$ above, and using the fact that both 4-dimensional components of $\mathcal{M}^{(3,1,1)}$ are isomorphic as $S_5$-modules, we conclude that \texttt{TT[30]} is the permutation group that we seek for our decomposition.  Computing characteristic polynomials of each side, we conclude that the characteristic polynomial of $\mathcal{S}^{(3,1,1)}$ on $C_{(n-2,2)}$ is divisible by exactly one power of $(x-1)$. 

The permutation module $\mathcal{M}^{(2,2,1)}$ has dimension 30 and decomposes as
\begin{align} \label{kostka}
\mathcal{M}^{(2,2,1)} = \mathcal{S}^{(2,2,1)} \oplus  \mathcal{S}^{(3,1,1)} \oplus 2\mathcal{S}^{(3,2)} \oplus 2\mathcal{S}^{(4,1)} \oplus \mathcal{S}^{(5)}.
\end{align}
There are three transitive permutation representations of $S_5$ on 30 points; if 

\noindent \texttt{TTT:=TransitiveGroups(30)}, then they are \texttt{TTT[22]}, \texttt{TTT[25]}, and \texttt{TTT[27]}.  A similar analysis to the ones above shows that \texttt{TTT[27]} is the representation we seek.  We find that the characteristic polynomial of $\mathcal{M}^{(2,2,1)}$ on $C_{(n-2,2)}$ has eigenvalue 1 of multiplicity 6. The decomposition (\ref{kostka}) has 7 summands, 6 of which have been shown to have eigenvalue 1 of multiplicity 1.  We conclude that $\mathcal{S}^{(2,2,1)}$ is not unisingular. 

As $n$ gets larger, the decompositions of $\mathcal{M}^{(2,2,1,\dots,1)}$ as in (\ref{kostka}) get more complicated, but this same strategy is, in principle, implementable.  For $n=15,\dots,23$, the online Atlas at \href{http://brauer.maths.qmul.ac.uk/Atlas/v3/}{http://brauer.maths.qmul.ac.uk/Atlas/v3/}, gives explicit matrix generators for the representations $\mathcal{S}^{(n-2,2)'}$ for $n=15,\dots,23$. 

\begin{rmk}
In \href{http://brauer.maths.qmul.ac.uk/Atlas/v3/}{http://brauer.maths.qmul.ac.uk/Atlas/v3/} it is not stated explicitly whether the precomputed representation of $S_n$ of degree $n(n-3)/2$ is $\mathcal{S}^{(n-2,2)}$ or $\mathcal{S}^{(n-2,2)'}$.   By comparing traces of elements against the known character tables of these groups, we verify that they are indeed $\mathcal{S}^{(n-2,2)'}$.   
\end{rmk}

For $n=5,7,9$ this procedure works since the \texttt{TransitiveGroups} database has all permutation groups of degree $\leq 48$.  For $n=11,13$ we construct the modules individually since they are out of range of the \texttt{TransitiveGroups} database. For $n=15,\dots,23$ we appeal to the representations precomputed in \href{http://brauer.maths.qmul.ac.uk/Atlas/v3/}{http://brauer.maths.qmul.ac.uk/Atlas/v3/}.  In all instances, the representations afford 1 as an eigenvalue on all classes except $C_{(n-2,2)}$, and on that class, evaluating the characteristic polynomial at 1 gives
\begin{center}
\begin{tabular}{c|c|c|c|c|c|c|c|c|c|c}
$n$ & 5 & 7 & 9 & 11 & 13 & 15  & 17 & 19 & 21 & 23 \\
\hline
$\chi_{(n-2,2)'}(1)$ &6&20&56&144& 352&832&1920&4352 & 9728 & 21504 
\end{tabular}
\end{center}
This pattern suggests that the character value on this class, if $n=2k+1$, is $2^{k-1}(2k-1)$.

\section{Concluding Remarks}

The paper \cite{cz} was concerned with specific families (certain groups of Lie type) and specific representations (Steinberg) that are unisingular.  We conclude by observing that many finite simple groups that are not covered by the results of \cite{cz} have unisingular representations over $\F_2$ as well.  We list just a few of them here, all of which have Schur indicator `+' or `$-$', and are therefore symplectic. However, none of these representations are candidates for a hyperelliptic unirealization over $\Q$ because the degree of the polynomial required does not divide the order of the group.  So if a hyperelliptic realization of any of these groups is possible, it would have to be over a nontrivial field extension of $\Q$ via a quotient of a larger group.  

\bigskip

\begin{center}
\begin{small}
\begin{tabular}{|c||c|c|c|c|c|c|c|c|c|c|c|}
\hline
Group & ${\rm L}_2(13)$ & ${\rm L}_3(3)$ & ${\rm U}_3(3)$ & $M_{11}$ & $M_{12}$ & $M_{22}$ & $\Aut(M_{22})$ & $M_{23}$ & $J_1$ & $J_2$ & ${\rm Sz}(8)$ \\
\hline
Dimension & 14 & 26 & 14 & 44 & 44, 144 & 34 & 34 & 120 & 76 & 36, 84 & 64 \\
\hline
\end{tabular}
\end{small}
\end{center}

\bigskip

\begin{rmk}
\begin{enumerate}
\item The groups $\Aut(M_{22})$ and ${\rm L}_3(4) \cdot 2_2$ have been shown to be Galois over $\Q$ in \cite{malle3}, and the Steinberg representation of ${\rm L}_3(4)$ of dimension 64 is unisingular by \cite{cz}).  However, the degrees of the polynomials in \cite{malle3} are not the ones required for hyperelliptic unirealization of these groups.  

More recently, in \cite{zywina}, Zywina showed that ${\rm L}_3(4)$ is Galois over $\Q$ using a degree-21 specialization $f(x)$ of Malle's parametric polynomials realizing ${\rm L}_3(4) \cdot 2_2$.  Since $\deg f = 21$, we get a 20 dimensional permutation representation
\[
{\rm L}_3(4) \hookrightarrow S_{21} \hookrightarrow \GL_{20}(\Z).
\]
This is precisely the \href{http://brauer.maths.qmul.ac.uk/Atlas/v3/matrep/L34G1-Zr20B0}{20-dimensional representation defined over $\Z$} in the online Atlas.  One easily checks that this representation is unisingular and irreducible over $\Q$. However, this representation decomposes over $\F_2$ as two 9-dimensional representations, and two 1-dimensionals.

\item The representation of ${\rm Sz}(8)$ listed in the table above is the Steinberg but, as recently as the preprint \cite{zywina}, it is is not known if ${\rm Sz}(8)$ is Galois over $\Q$, let alone unirealizable. 
\end{enumerate}
\end{rmk}

\end{document}